\documentclass[12pt,a4paper]{article}
\usepackage{tikz}
\usepackage[OT4]{fontenc}
\usepackage[T1]{fontenc}
\usepackage{amsfonts}
\usepackage{amsmath}
\usepackage{amsthm}
\usepackage{epic}
\usepackage{eepic}
\usepackage{color}
\usepackage{pst-node}

\newtheorem{thm}{Theorem}
\newtheorem{cor}{Corollary}

\newtheorem{df}{Definition}

\newtheorem{prop}{Proposition}
\newtheorem{lem}{Lemma}

\newtheorem{example}{Example}
\theoremstyle{remark}

\begin{document}

\begin{center}
\noindent\textbf{\Large Bipartization of graphs}
\end{center}

\begin{center}{Mateusz Miotk, Jerzy Topp, and Pawe\l{}~\.Zyli\'nski\\[1mm]
University of Gda\'{n}sk, 80-952 Gda\'nsk, Poland\\
{\small \texttt{\{mmiotk,j.topp,zylinski\}@inf.ug.edu.pl}}}\end{center}

\date{}

\begin{abstract}
\noindent
A dominating set of a graph $G$ is a set $D\subseteq V_G$ such that every vertex in $V_G-D$ is adjacent to at least one vertex in $D$, and the domination number $\gamma(G)$ of $G$ is the minimum cardinality of a dominating set of $G$.  In this paper we provide a~new characterization of bipartite graphs whose domination number is equal to the cardinality of its smaller partite set. Our characterization is based upon a~new graph operation.\end{abstract}


\section{Introduction and notation}\label{sec:intro} For notation and graph theory terminology we in general follow \cite{ChLZ15}. Specifically, let $G=(V_G,E_G)$ be a graph with vertex set $V_G$ and edge set $E_G$. For a subset $X \subseteq~V_G$, the   {\em subgraph induced} by $X$ is denoted by $G[X]$. For simplicity of notation, if $X=\{x_1,\ldots, x_k\}$, we shall write $G[x_1,\ldots, x_k]$ instead of $G[\{x_1,\ldots, x_k\}]$. For a~vertex $v$ of $G$, its {\em neighborhood\/}, denoted by $N_{G}(v)$, is the set of all vertices adjacent to $v$, and the cardinality of $N_G(v)$, denoted by $\deg_G(v)$, is called the {\em degree} of~$v$. The {\em closed neighborhood\/} of $v$, denoted by $N_{G}[v]$, is the set $N_{G}(v)\cup \{v\}$.  In general, the {\em  neighborhood\/} of $X \subseteq V_G$, denoted by $N_{G}(X)$, is defined to be $\bigcup_{v\in X}N_{G}(v)$, and the {\em closed\/} neighborhood of $X$, denoted by $N_{G}[X]$, is the set $N_{G}(X)\cup X$. A vertex of degree one is called a {\em leaf}, and the only neighbor of a leaf is called its {\em support vertex} (or simply, its {\em support}). A {\em weak support} is a vertex adjacent to exactly one leaf. Finally, the set of leaves and the set of supports of $G$ we denoted by~$L_G$ and $S_G$, respectively.

A subset $D$ of $V_G$ is said to be a {\em dominating set} of a graph $G$ if each vertex belonging to the set $V_G -D$ has a~neighbor in $D$. The cardinality of a~minimum dominating set of $G$ is called the~{\em domination number of $G$} and is denoted by $\gamma(G)$. A subset $C \subseteq V_G$ is a~{\em covering set} of $G$ if each edge of $G$ has an end-vertex in $C$. The cardinality of a~minimum covering set of $G$ is called the~{\em covering number of $G$} and denoted by $\beta(G)$.

It is obvious that if $G=((A,B),E_G)$ is a bipartite graph, then $\gamma(G)\le \min\{|A|,|B|\}$. In this paper the set of all bipartite graphs $G=((A,B),E_G)$ in which $\gamma(G)= \min\{|A|,|B|\}$ is denoted by ${\cal B}$. Some properties of the graphs belonging to the set  ${\cal B}$ were observed in the papers \cite{AJBT13,HR95,MTZxx,RV98,WY12}, where all graphs with the domination number equal to the covering number were characterized. In this paper, inspired by results and constructions of Hartnell and Rall \cite{HR95}, we introduce a~new graph operation, called the {\em  bipartization of a graph with respect to a~function\/}, study basic properties of this operation, and provide a new characterization of the graphs belonging to the set ${\cal B}$ in terms of this new operation.

\section{Bipartization of a graph}\label{sec:bip}
Let ${\cal K}_H$ denote the set of all complete subgraphs of a graph $H$.
If $v\in V_H$, then the set $\{K\in {\cal K}_H \colon v \in V_K\}$ is denoted
by ${\cal K}_H(v)$. If $X\subseteq V_H$, then the set $\bigcup_{v\in X}{\cal K}_H(v)$ is denoted by ${\cal K}_H(X)$, and it is obvious that ${\cal K}_H(X)= \{K\in {\cal K}_H
\colon V_K\cap X\not=\emptyset\}$. Let $f\colon {\cal K}_H\to \mathbb{N}$ be a function. If $K\in {\cal K}_H$, then by ${\cal F}_K$ we denote the set $\{(K,1),\ldots, (K,f(K))\}$ if $f(K)\ge 1$, and we let ${\cal F}_K=\emptyset$ if $f(K)=0$. By ${\cal K}_H^f$ we denote the set of all positively $f$-valued complete subgraphs of $H$, that is, ${\cal K}_H^f=\{K\in {\cal{K}}_H\colon f(K)\ge 1\}$.

\begin{df} Let $H$ be a graph and let $f\colon {\cal K}_H\to \mathbb{N}$ be a function.
The {\em bipartization} of $H$ with respect to $f$ is the bipartite graph $B_f(H) =((A,B),E_{B_f(H)})$ in which $A=V_H$, $B=\bigcup_{K\in {\cal K}_H}{\cal F}_K$, and where a vertex $x\in A$ is adjacent to a vertex $(K,i)\in B$ if and only if $x$ is a vertex of the complete graph $K$ $($$i=1,\ldots, f(K)$$)$. \end{df}

\begin{example} \label{przyklad-1}  {\rm Fig.~\ref{rys-grafu-H_f} presents a graph $H$ (for which ${\cal K}_H=\{H[a], H[b], H[c], H[d], H[a, b],$ $H[a,c], H[b,c], H[c,d], H[a,b,c]\}$) and its two bipartizations $B_f(H)$ and $B_g(H)$ with respect to functions $f,\,g\colon {\cal K}_H\to \mathbb{N}$, respectively, where $f(H[a])=1$, $f(H[b])=1$, $f(H[c])=2$, $f(H[d])=0$, $f(H[a,b])=3$, $f(H[a,c])=0$, $f(H[b,c])=2$, $f(H[c,d])=3$, $f(H[a, b, c])=1$, while $g(H[v])=0$ for every vertex $v \in V_H$, $g(H[u,v])=1$ for every edge $uv \in E_H$, and $g(H[a,b,c]) = 0$. Observe that $B_g(H)$ is the subdivision graph $S(H)$ of $H$ (i.e., the graph obtained from $H$ by inserting a new vertex into each edge of $H$).} \end{example}

\begin{figure}[!h]
\begin{center}\special{em:linewidth 0.4pt}
\unitlength 0.25ex \linethickness{0.4pt}\begin{picture}(260,150)
\put(-20,30){\path(0,0)(60,0)(30,52)(0,0)\path(30,52)(30,112)\put(2,95){${}^{H}$}
\multiput(0,0)(60,0){2}{\circle*{2.5}}\multiput(30,52)(0,60){2}{\circle*{2.5}}
\put(-2,-9){${}^{a}$}\put(59,-10){${}^{b}$}\put(24,48){${}^{c}$}\put(24,107){${}^{d}$}}
\put(100,30){\path(60,0)(0,0)\path(30,52)(30,112)\path(0,0)(30,17)(60,0) \path(30,17)(30,52)\put(-10,95){${}^{B_f(H)}$}\path(-25,0)(0,0)\path(85,0)(60,0)
\multiput(0,0)(60,0){2}{\circle*{2.5}}\multiput(30,52)(0,60){2}{\circle*{2.5}}
\path(0,0)(30,-10)(60,0)\path(0,0)(30,-20)(60,0)\path(30,52)(20,82)(30,112) \path(30,52)(40,82)(30,112)\path(30,52)(40,23)(60,0)\path(30,52)(50,29)(60,0)
\path(53,44)(30,52)(53,60)\put(-2,-9){${}^{a}$}\put(60,-10){${}^{b}$}\put(24,48){${}^{c}$}
\put(23,107){${}^{d}$}\put(55,55){${}^{(H[c],1)}$}\put(55,39){${}^{(H[c],2)}$}
\put(-14,35){${}^{(H[a,b,c],1)}$}\path(-14,35)(18,35)(28,20)\put(-38,1){${}^{(H[a],1)}$}
\put(74,1){${}^{(H[b],1)}$}\put(28,20){\vector(1,-2){0}}
\put(52,24){${}^{(H[b,c],1)}$}\put(69,12){${}^{(H[b,c],2)}$}\path(96,12)(69,12)(42,22)
\put(42,22){\vector(-2,1){0}}\put(16,0){${}^{(H[a,b],1)}$}\put(46,-20){${}^{(H[a,b],2)}$}
\path(74,-20)(47,-20)(32,-11)\put(32,-11){\vector(-2,1){0}}\put(16,-30){${}^{(H[a,b],3)}$}
\put(-10,77){${}^{(H[c,d],1)}$}\put(42,89){${}^{(H[c,d],2)}$}
\path(32,83)(43,89)(70,89)\put(32,83){\vector(-2,-1){0}}\put(42,77){${}^{(H[c,d],3)}$}
\put(30,17){\whiten\circle{2.5}}\put(-25,0){\whiten\circle{2.5}}
\put(85,0){\whiten\circle{2.5}}\multiput(30,0)(0,-10){3}{\whiten\circle{2.5}}
\multiput(20,82)(10,0){3}{\whiten\circle{2.5}}\put(40,23){\whiten\circle{2.5}}
\put(50,29){\whiten\circle{2.5}}\put(53,44){\whiten\circle{2.5}}\put(53,60){\whiten\circle{2.5}}
\put(120,0){\path(0,0)(60,0)(30,52)(0,0)\path(30,52)(30,112)
           \put(-3,95){${}^{B_g(H)}$}\multiput(0,0)(60,0){2}{\circle*{2.5}}
           \multiput(30,52)(0,60){2}{\circle*{2.5}}
           \multiput(30,0)(0,82){2}{\whiten\circle{2.5}}
           \multiput(15,26)(30,0){2}{\whiten\circle{2.5}}
           \put(-2,-9){${}^{a}$}\put(59,-10){${}^{b}$}\put(24,48){${}^{c}$}
           \put(24,107){${}^{d}$}\put(-3,77){${}^{(H[c,d],1)}$}
           \put(-18,21){${}^{(H[a,c],1)}$}\put(50,21){${}^{(H[b,c],1)}$}
           \put(16,-13){${}^{(H[a,b],1)}$} }
}\end{picture} \caption{Graphs $H$, $B_f(H)$, and $B_g(H)$.} \label{rys-grafu-H_f} \end{center}\end{figure}
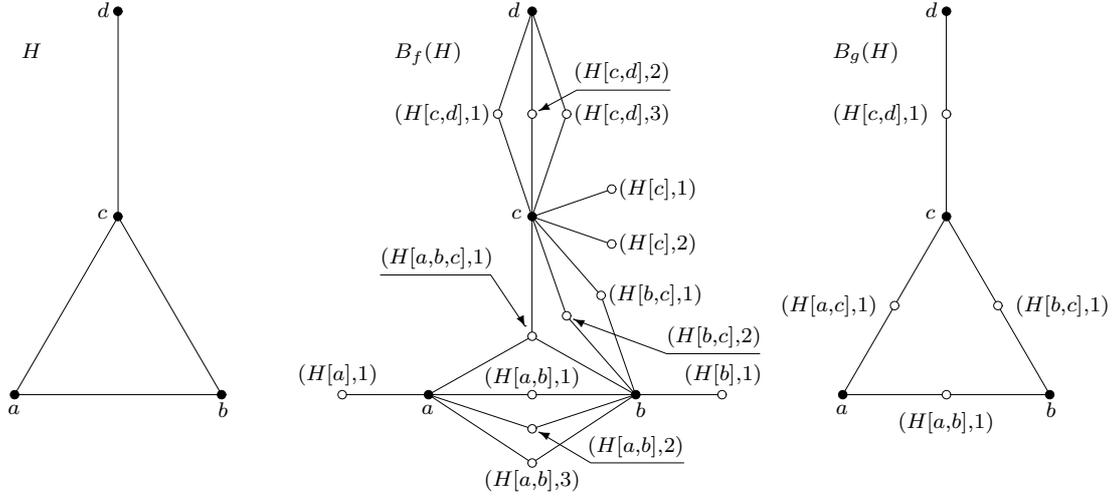

\section{Properties of bipartizations of graphs}

It is clear from the above definition of the bipartization of a graph with respect to a~function that we have the following proposition.

\begin{prop} \label{Proposition-first} The bipartization of a graph with respect to a~function has the following properties:
\begin{itemize} \item[$(1)$] If $B_f(H)=((A,B),E_{B_f(H)})$ is the bipartization of a graph $H$ with respect to a~function $f\colon {\cal K}_H\to \mathbb{N}$, then: \begin{itemize}
    \item[$(\rm a)$] $N_{B_f(H)}(v)= \bigcup_{K\in {\cal K}_H(v)}{\cal F}_K$ if $v\in A$. \item[$(\rm b)$] $N_{B_f(H)}(X)= \bigcup_{K\in {\cal K}_H(X)}{\cal F}_K$ if $X\subseteq A$. \item[$(\rm c)$] $N_{B_f(H)}((K,i))= V_K$ if $(K,i)\in B$ $($$i=1,\ldots, f(K)$$)$. \item[$(\rm d)$] $|V_{B_f(H)}| = |V_H| + \sum_{K \in {\cal K}_H} f(K)$ and $|E_{B_f(H)}| = \sum_{K \in {\cal K}_H} f(K) \, |V_K|$. \end{itemize}
\item[$(2)$] If $H$ is a connected graph and $f \colon {\cal{K}}_H \to \mathbb{N}$ is a function such that every edge of $H$ belongs to a positively $f$-valued complete subgraph of $H$, then the bipartization $B_f(H)$ is a connected graph.
\item[$(3)$] If $H$ is a graph and $f,\, g \colon {\cal K}_H\to \mathbb{N}$ are functions such that $f(K) \ge g(K)$ for every $K\in {\cal K}_H$, then the graph $B_g(H)$ is an induced subgraph of $B_f(H)$.
\end{itemize}\end{prop}

Our study of properties of bipartizations we begin by showing that every bipartite graph is the bipartization of some graph with respect to some function.

\begin{thm}\label{thm:bip_existence} For every bipartite graph $G=((A,B),E_G)$ there exist a graph $H$ and a~function $f\colon {\cal K}_H\to \mathbb{N}$ such that $G=B_f(H)$. \end{thm}

\begin{proof} We say that vertices $x$ and $y$ of $G$ are {\em similar\/} if $N_G(x)=N_G(y)$. It is obvious that this similarity is an equivalence relation on $B$ (as well as on $A$ and $A \cup B$). Let $B_1,\ldots, B_l$ be the equivalence classes of this  relation on $B$, say $B_i= \{b_1^i,b_2^i,\ldots,b_{k_i}^i\}$ for $i=1,\ldots, l$. It follows from  properties of the equivalence classes that $|B_1|+\ldots+|B_l|= |B|$, $N_G(b_1^i)=N_G(x)$ for every $x\in B_i$, and $N_G(b_1^i)\not =N_G(b_1^j)$ if $i, j\in\{ 1,\ldots,l\}$ and $i\not=j$.

Now, let $H=(V_{H},E_{H})$ be a graph in which $V_H=A$ and two vertices $x$ and $y$ are adjacent in $H$ if and only if they are at distance two apart from each other in $G$. Let ${\cal K}_H$ be the set of all complete subgraphs  of $H$, and let $f\colon {\cal K}_H\to \mathbb{N}$ be a function such that $f(K)= |\{b\in B\colon N_G(b)=V_K\}|$ for $K\in {\cal K}_H$. Next, let $K_i$ be the induced subgraph $H[N_G(b_1^i)]$ of $H$. It follows from the definition of $H$ that $K_i$ is a complete subgraph of $H$. In addition, from the definition of $f$ and from properties of the classes $B_1,\ldots, B_l$, it follows that $f(K_i)=|B_i|>0$ ($i=1,\ldots,l$), and $f(K)=0$ if $K\in {\cal K}_H-\{K_1,\ldots, K_l\}$. Consequently, ${\cal K}_H^f=\{ K_1,\ldots,K_l\}$.

Finally, consider the bipartite graph $B_f(H)=((X,Y),E_{B_f(H)})$ in which  $X=V_H=A$, $Y= \bigcup_{K\in {\cal K}_H}{\cal F}_K = \bigcup_{K\in {\cal K}_H^f}{\cal F}_K = \bigcup_{i=1}^{l} \{(K_i,1),\ldots,(K_i,k_i)\}$, and where $N_{B_f(H)}((K_i,j))$ $= V_{K_i}=N_G(b_1^i)$ for every $(K_i,j)\in Y$. Now, one can observe that the function $\varphi\colon A\cup B\to X\cup Y$, where $\varphi(x)=x$ if $x\in A$, and $\varphi(b_j^i)= (K_i,j)$ if $b_j^i\in B$, is an isomorphism between graphs $G$ and $B_f(H)$. \end{proof}

We have proved that a bipartite graph $G=((A,B),E_G)$ is the bipartization $B_f(H)$
of a~graph $H=(V_H,E_H)$ (in which $V_H=A$ and $E_H=\{xy\colon \,x,y \in A\,\, \mbox{and}\,\, d_G(x,y)=2\}$) with respect to a function $f\colon {\cal K}_H\to \mathbb{N}$, where $f(K)= |\{b\in B\colon N_G(b)=V_K\}|$ for $K\in {\cal K}_H$. The same graph $G$ is also the bipartization $B_g(F)$ of a graph $F=(V_F,E_F)$ (in which $V_F=B$ and $E_F=\{xy\colon \,x,y \in B\,\, \mbox{and}\,\, d_G(x,y)=2\}$) with respect to a function $g\colon {\cal K}_F\to \mathbb{N}$, where $g(K)= |\{a\in A\colon N_G(a)=V_K\}|$ for $K\in {\cal K}_F$. Consequently, every bipartite graph may be the bipartization of two non-isomorphic graphs.

\begin{example} \label{przyklad-2} {\rm Fig.~\ref{rys-grafu-bipatrization-dwa+} depicts the bipartite graph $G$ which is the bipartization of the non-isomorphic graphs $H$ and $F$ with respect to functions $\overline{f}\colon {\cal K}_H \to \mathbb{N}$ and $\overline{g}\colon {\cal K}_F\to \mathbb{N}$, respectively, which non-zero values are displayed in the figure. } \end{example}

\begin{figure*}[!h] \begin{center}
\special{em:linewidth 0.4pt}
\unitlength 0.25ex \linethickness{0.4pt}\begin{picture}(260,125)
\put(-20,10){\special{em:linewidth 0.4pt}
\unitlength 0.25ex \linethickness{0.4pt}\begin{picture}(80,110)
\path(5,10)(40,30)(58,20)(58,0)\path(58,20)(75,30)
\path(40,30)(40,90)(22,100)\path(40,90)(58,100)
\multiput(22,20)(36,0){2}{\circle*{2.5}}\multiput(40,50)(0,40){2}{\circle*{2.5}}
\multiput(40,30)(0,40){2}{\whiten\circle{2.5}}\multiput(22,100)(36,0){2}{\whiten\circle{2.5}}
\put(5,10){\whiten\circle{2.5}}\put(58,0){\whiten\circle{2.5}}\put(75,30){\whiten\circle{2.5}}
\put(3,1){${}^{x}$}\put(17,95){${}^{t}$}\put(20,11){${}^{a}$}\put(37.25,21){${}^{w}$}
\put(56.5,20){${}^{b}$}\put(60.5,-5){${}^{z}$}\put(73,20.5){${}^{y}$}\put(42,45){${}^{c}$}
\put(42,65){${}^{u}$}\put(42,83){${}^{d}$}\put(60.5,94.5){${}^{v}$}
\put(37,-5){${}^{G}$}\end{picture}}
\put(80,50){\special{em:linewidth 0.4pt}
\unitlength 0.25ex \linethickness{0.4pt}\begin{picture}(80,80)
\put(0,-40){\path(40,90)(40,50)(22,20)(58,20)(40,50)
\multiput(40,50)(0,40){2}{\circle*{2.5}}\multiput(22,20)(36,0){2}{\circle*{2.5}}
\put(19.5,11){${}^{a}$}\put(56.5,10){${}^{b}$}\put(42.5,48){${}^{c}$}\put(42.5,85){${}^{d}$}
\put(37,-5){${}^{H}$}\put(55,80){${}^{\overline{f}(H[a])=1}$}
\put(55,70){${}^{\overline{f}(H[b])=2}$}\put(55,60){${}^{\overline{f}(H[d])=2}$}
\put(55,50){${}^{\overline{f}(H[c,d])=1}$}\put(55,40){${}^{\overline{f}(H[a,b,c])=1}$}} \end{picture}}
\put(180,10){\special{em:linewidth 0.4pt}
\unitlength 0.25ex \linethickness{0.4pt}\begin{picture}(80,80)
\path(5,10)(40,30)(75,30)(58,0)(40,30)(40,70)(22,100)(58,100)(40,70)
\multiput(22,100)(36,0){2}{\whiten\circle{2.5}}\multiput(40,30)(0,40){2}{\whiten\circle{2.5}}
\put(5,10){\whiten\circle{2.5}}\put(58,0){\whiten\circle{2.5}}\put(75,30){\whiten\circle{2.5}}
\put(17,94.5){${}^{t}$}\put(60.5,94.5){${}^{v}$}\put(42,65){${}^{u}$}
\put(42,29.5){${}^{w}$}\put(60.5,-5){${}^{z}$}\put(3,1){${}^{x}$}\put(73,31){${}^{y}$}
\put(37,-5){${}^{F}$}\put(60,80){${}^{\overline{g}(F[x,w])=1}$}
\put(60,70){${}^{\overline{g}(F[u,w])=1}$}\put(60,60){${}^{\overline{g}(F[t,v,u])=1}$}
\put(60,50){${}^{\overline{g}(F[w,y,z])=1}$}\end{picture}}
\end{picture} \caption{Graph $G$ is the bipartization of the two non-isomorphic graphs $H$ and $F$.} \label{rys-grafu-bipatrization-dwa+} \end{center}
\end{figure*}
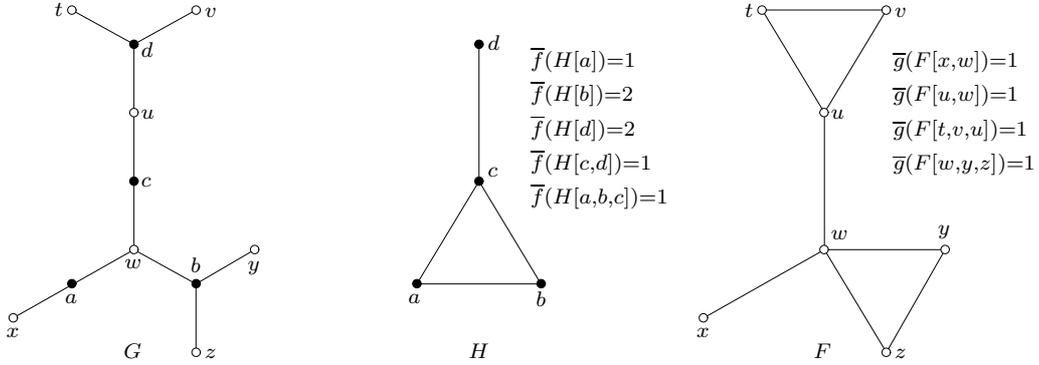

It is obvious from Theorem \ref{thm:bip_existence} that every tree is a bipartization.
We are now interested in providing a simple characterization of graphs $H$ and functions $f\colon {\cal K}_H \to \mathbb{N}$ for which the bipartization $B_f(H)$ is a tree.
We begin with the following notation: An alternating sequence of vertices and complete graphs $(v_0,F_1,v_1, \ldots, v_{k-1},F_k,v_k)$  is said to be a~{\it positively $f$-valued complete $v_0-v_k$ path} if $v_{i-1}v_i$ is an edge in the complete graph $F_i$ for $i=1,\ldots, k$. We now have the following two useful lemmas.

\begin{lem}\label{lem:tree}
Let $H$ be a connected graph, and let $f \colon {\cal{K}}_H \to \mathbb{N}$ be a function. If there are two vertices $u$ and $v$ and two distinct internally vertex-disjoint positively $f$-valued complete $u-v$ paths in $H$,  then the bipartization $B_f(H)$ contains a cycle. \end{lem}
\begin{proof} If $(v_0=u, F_1,v_1,\ldots, v_{m-1},F_m,v_m=v)$ and $(v_0'=u, F_1', v_1',\ldots, v_{n-1}',F_n',v_n'=v)$ are distinct internally vertex-disjoint positively $f$-valued complete $u-v$ paths in $H$, then $(v_0, (F_1,1),v_1,\ldots, v_{m-1},(F_m,1), v_m)$ and $(v_0', (F_1',1),v_1',\ldots, v_{n-1}',(F_n',1),v_n')$ are distinct $u-v$ paths in $B_f(H)$, and so they generate at least one cycle in $B_f(H)$. \end{proof}

Let us recall first that a maximal connected subgraph without a cutvertex is called a~{\em block}. A graph $H$ is said to be a {\em block graph} if each block of $H$ is a complete graph. The next lemma is probably known, therefore we omit its easy inductive proof.

\begin{lem}\label{lem:block-graphs} If ${\cal S}$ is the set of all blocks of a graph $H$, then $\sum\limits_{B\in {\cal S}}\left(|V_B|-1\right)= |V_H|-1$. \end{lem}

Now we are ready for a characterization of graphs which bipartizations
(with respect to some functions) are trees.

\begin{thm}\label{thm:bip_tree} Let $H$ be a connected graph, and let $f \colon {\cal{K}}_H \to \mathbb{N}$ be a~function such that every edge of $H$ belongs to some
positively $f$-valued complete subgraph of $H$. Then the bipartization $B_f(H)$ is a tree if and only if the following conditions hold:
\begin{itemize}
\item[$(1)$] $f(K) \le 1$ for every non-trivial complete subgraph $K$ of $H$.
\item[$(2)$] $H$ is a block graph.
\item[$(3)$] For a non-trivial complete subgraph $K$ of $H$ is $f(K)=1$ if and
only if $K$ is a~block of $H$.
\end{itemize}
\end{thm}

\begin{proof} Assume that $B_f(H)$ is a tree. The statement (1) is obvious, for if there
were a~non-trivial complete subgraph $K$ of $H$ for which $f(K)\ge 2$, then for any
two vertices $u$ and $v$ belonging to $K$, the sequence $(u,(K,1),v,(K,2),u)$ would
be a cycle in $B_f(H)$.

Suppose now that $H$ is not a block graph. Then there exists a~block in $H$, say $B$, which is not a complete graph. Thus in $B$ there exists a~cycle such that not all its chords belong to $B$. Let $C=(v_0,v_1,\ldots, v_l,v_0)$ be a~shortest such cycle in $B$. Then $l\ge 3$ and we distinguish two cases. If $C$ is chordless, then, by Lemma \ref{lem:tree}, $B_f(H)$ contains a~cycle. Thus assume that $C$ has a chord. We may assume that $v_0$ is an end-vertex of a chord of $C$, and then let $k$ be the smallest integer such that $v_0v_k$ is a chord of $C$. Now the choice of $C$ implies that the vertices $v_0, v_1, \ldots, v_k$ are mutually adjacent, and therefore, $k=2$. Similarly, $v_0, v_k, \ldots, v_l$ are mutually adjacent, and so we must have $l=3$. Consequently, $C=(v_0,v_1,v_2,v_3,v_0)$ and $v_0v_2$ is the only chord of $C$. Now it is obvious that there are at least two $v_0-v_2$ positively $f$-valued complete paths in $H$. From this and from Lemma \ref{lem:tree} it follows that the bipartition $B_f(H)$ contains a cycle. This contradiction completes the proof of the statement~(2).

Let $B$ be a block of $H$. We have already proved that $B$ is a complete graph. Let $B'$ be a proper non-trivial complete subgraph of $B$. To prove (3), it suffices to observe that $f(B')=0$. On the contrary, suppose that $f(B')\not=0$. We now choose two distinct vertices $v$ and $u$ belonging to $B'$, and a vertex $w$ belonging to $B$ but not to $B'$. This clearly forces that there are at least two $v-u$ positively $f$-valued complete paths in $H$. Consequently, by Lemma \ref{lem:tree}, $B_f(H)$ contains a~cycle, and this contradiction completes the proof of the statement~(3).

Assume now that the conditions (1)--(3) are satisfied for $H$ and $f$. Since end-vertices of $B_f(H)$, corresponding to positively $f$-valued one-vertex complete subgraphs of $H$, are not important to our study of tree-like structure of $B_f(H)$, we can assume without loss of generality that $f(H[v])=0$ for every vertex $v \in V_H$. Consequently, $H$ is a~block graph and $f(K)=1$ for every block $K$ of $H$, while $f(K')=0$ for every other complete subgraph $K'$ of $H$. It remains to prove that $B_f(H)$ is a tree. Since $B_f(H)$ is a connected graph, it suffices to show that $|E_{B_f(H)}| = |V_{B_f(H)}|-1$. Let ${\cal S}$ be the set of all blocks of $H$. Then ${\cal K}_H^f={\cal S}$,  $|V_{B_f(H)}|= |V_H|+ \sum_{K\in {\cal K}_H^f}\!f(K)=  |V_H|+|{\cal S}|$, and $|E_{B_f(H)}|=  \sum_{K\in {\cal K}_H^f}\!f(K)|V_K|= \sum_{K\in {\cal S}}|V_K|= \sum_{K\in {\cal S}}(|V_K|-1)+|{\cal S}|$. Now, since $\sum_{K\in {\cal S}}(|V_K|-1) =|V_H|-1$ (by Lemma \ref{lem:block-graphs}), we finally have $|E_{B_f(H)}| = (|V_{H}|-1)+|{\cal S}|= (|V_{H}|+|{\cal S}|)-1= |V_{B_f(H)}|-1$. \end{proof}

\begin{cor}
For every connected graph $H$, there exists a function $f \colon {\cal K}_{H} \to \mathbb{N}$ such that the bipartization $B_f(H)$ is a tree.
\end{cor}

\begin{proof}
Let $F$ be a spanning block graph of $H$ and let $f\colon {\cal K}_F \rightarrow \{0,1\}$ be a function such that $f(K)=1$ if and only if $K$ is a block of $F$. Clearly, $f$ satisfies the conditions (1)--(3) of Theorem \ref{thm:bip_tree}, and so the bipartization $B_f(H)$ is a tree. \end{proof}

\begin{example} \label{przyklad-3} {\rm Fig.~\ref{rys-grafu-bipatrization-dwa+} shows the tree $G$ which is the bipartization of two block graphs $H$ and $F$ with respect to functions $\overline{f}$ and $\overline{g}$, respectively, which non-zero values are listed  in the same figure.}\end{example}

\section{Graphs belonging to the family ${\cal B}$}\label{sec:dom}
In this section, we provide an alternative characterization of all bipartite graphs whose domination number is equal to the cardinality of its smaller partite set, that is, we prove that a graph $G$ belongs to the class ${\cal B}$ if and only if $G$ is some bipartization of a~graph. For that purpose, we need the following lemma.

\begin{lem}\label{thm:main2}{\em\cite{MTZxx}}
 Let $G=((A,B),E_G)$ be a connected bipartite graph with $1\le |A|\le |B|$. Then the following statements are equivalent:
\begin{itemize}
\item[$(1)$] $\gamma(G)=|A|$.
\item[$(2)$] $\gamma(G)=\beta(G)=|A|$.
\item[$(3)$] $G$ has the following two properties:
\begin{itemize}
\item[{\rm (a)}] Each support vertex of $G$ belonging to $B$ is a weak support and each of its non-leaf neighbors is a support.
\item[{\rm (b)}]
If $x$ and $y$ are vertices belonging to $A-(L_G\cup S_G)$ and $d_G(x,y)=2$, then there are at least two vertices $\overline{x}$ and $\overline{y}$ in $B$ such that $N_G(\overline{x})= N_G(\overline{y})= \{x,y\}$.
\end{itemize}
\end{itemize} \end{lem}

We are ready to establish our main theorem that provides an alternative characterization of the graphs belonging to ${\cal B}$ in terms of the bipartization of a graph.

\begin{thm}\label{thm:dom_main2} Let $G=((A,B),E_G)$ be a connected bipartite graph with 
$1\le |A| \le |B|$. Then $\gamma(G) =|A|$ if and only if $G$ is the bipartization $B_f(H)$ of a connected graph $H$ with respect to a non-zero function $f\colon {\cal K}_H \to \mathbb{N}$ and $f$ has the following two properties: \begin{itemize} \item[$(1)$] If $uv\in E_H$ and $f(H[u,v]) =0$, then $f(H')>0$ for some complete subgraph $H'$ of $H$ containing the edge $uv$. \item[$(2)$] If $uv\in E_H$ and $f(H[u])=f(H[v])=0$, then $f(H[u,v])\ge 2$. \end{itemize}
\end{thm}

\begin{proof}
Assume first that $\gamma(G) = |A|$. Then $G$ has the properties (3a) and (3b) of Lemma \ref{thm:main2}. Let $H=(V_H,E_H)$ be a graph in which $V_H=A$ and $E_H=\{xy \colon x,y \in A \textrm{ and } d_G(x,y)=2\}$, and let $f\colon {\cal K}_H \to \mathbb{N}$ be a function such that $f(K)= |\{x\in B\colon N_G(x)=V_K\}|$ for each $K\in {\cal K}_H$. Then $G$ is the bipartization $B_f(H)$ of $H$ with respect to $f$, as we have shown in the proof of Theorem \ref{thm:bip_existence}. It is obvious that if $H=K_1$, then ${\cal K}_H=\{H\}$ and it must be $f(H)\ge 1$ (as otherwise $G=B_f(H)$ would be a graph of order one). Thus assume that $H$ is non-trivial. Now it remains to prove that $f$ has the properties (1) and (2).

Let $uv$ be an edge of $H$ such that $f(H[u,v])=0$. Suppose on the contrary that $f(H')=0$ for every complete subgraph $H'$ containing the edge $uv$. Then the vertices $u$ and $v$ do not share a neighbor in $B_f(H)=G$, so $d_G(u,v) > 2$ and $uv$ is not an edge in $H$, a~contradiction. This proves the property (1).

Now let $uv$ be an edge of $H$ such that $f(H[u])=f(H[v])=0$. From these assumptions it follows that $d_G(u,v)=2$ and neither $u$ nor $v$ is a support vertex in $G=B_f(H)$. Now we shall prove that none of the vertices $u$ and $v$ is a leaf in $G$. First, because $u, v\in A$ and they have a common neighbor, it follows from the first part of the property (3a) of Lemma \ref{thm:main2} that at least one of the vertices $u$ and $v$ is not a leaf in $G$. Suppose now that exactly one of the vertices $u$ and $v$ is a leaf in $G$, say $u$ is a leaf. Then it follows from the second part of the property (3a) of Lemma \ref{thm:main2} that $v$ is a support vertex in $G=B_f(H)$ and, therefore, $f(H[v])>0$, a contradiction. Consequently, both $u$ and $v$ are elements of $A-N_G[L_G]$. Thus, since $d_G(u,v)=2$, the property (3b) of Lemma \ref{thm:main2} implies that there are at least two vertices $\bar{u}, \bar{v} \in B$ such that $N_{G}(\bar{u})=N_{G}(\bar{u})=\{u,v\}$. Therefore $f(H[u,v])= |\{x\in B\colon N_G(x)=\{u, v\}\}|\ge |\{\bar{u}, \bar{v}\}|=2$ and this proves the property (2).

Assume now that $H$ is a connected graph, and $f\colon {\cal K}_H \to \mathbb{N}$ is a non-zero function having the properties (1) and (2). We shall prove that in the bipartization $B_f(H)=((A,B),E_{B_f(H)})$, where $A=V_H$ and $B=\bigcup_{K\in {\cal K}_H} {\cal F}_K$, is $|A|\le |B|$ and $\gamma(B_f(H))=|A|$. This is obvious if $H$ is a graph of order 1. Thus assume that $H$ is a graph of order at least 2. From the property (1) it follows that $B_f(H)$ is a connected graph. We first prove the inequality $|A|\le |B|$. To prove this, it suffices to show that $B_f(H)$ has an $A$-saturating matching. We begin by dividing $A=V_H$ into two subsets $V_H^1= \{v\in V_H\colon f(H[v])\ge 1\}$ and $V_H^0= \{v\in V_H\colon f(H[v])=0\}$. It is obvious that the edge-set $M^1= \{v(H[v],1) \colon v\in V_H^1\}$ is a $V_H^1$-saturating matching in $B_f(H)$. Next, we order the set $V_H^0$ in an arbitrary way, say $V_H^0=\{v_1,\ldots,v_n\}$. Now, depending on this order, we consecutively choose edges $e_1, \ldots, e_n$ in such a way that $M^1\cup \{e_1,\ldots, e_i\}$ is a~$(V_H^1\cup \{v_1,\ldots,v_i\})$-saturating matching in $B_f(H)$.

Assume that we have already chosen a $(V_H^1\cup \{v_1,\ldots,v_{i-1}\})$-saturating matching $M^1\cup\{e_1,\ldots, e_{i-1}\}$ in $B_f(H)$, and consider the next vertex $v_i \in V_H^0$. If $N_{H}(v_i)\cap V_H^0\not=\emptyset$, say $v_j\in N_{H}(v_i)\cap V_H^0$, then $f(H[v_j])=0$ and therefore $f(H[v_i,v_j])\ge 2$ (by the property (2)) and the edge $e_i= v_i(H[v_i,v_j],1)$ if $j>i$ ($e_i= v_i(H[v_i,v_j],2)$ if $j<i$) together with $M^1\cup \{e_1,\ldots, e_{i-1}\}$ form a  $(V_H^1\cup \{v_1,\ldots,v_i\})$-saturating matching in $B_f(H)$. Thus assume that $N_H(v_i) \subseteq V_H^1$. Let $v$ be a neighbor of $v_i$ in $H$. If $f(H[v_i,v])\ge 1$, then the edge $e_i= v_i(H[v_i,v],1)$ has the desired property. Finally, if $f(H[v_i,v])=0$, then $f(H')>0$ for some complete subgraph $H'$ of $H$ containing the edge $v_iv$ (by the property (1)) and in this case the edge $e_i= v_i(H',1)$ has the desired property (as $N_H(v_i) \subseteq V_H^1$). Repeating this procedure as many times as needed, an $A$-saturating matching in $B_f(H)$ can be obtained.

To complete the proof, it remains to show that  $\gamma(B_f(H)) = |A|$. In a standard way, suppose to the contrary that $\gamma(B_f(H)) < |A|$. Let $D$ be a minimum dominating set of $B_f(H)$ with  $|D \cap A|$ as large as possible. Since $\gamma(B_f(H))=|D|$, the inequality $\gamma(B_f(H))<~|A|$ implies that $|A - D|> |D \cap B| \ge 1$. In addition, since $|D \cap A|$ is as large as possible, the set $V_H^1$ $(= \{v\in V_H\colon f(H[v])\ge 1\})$ is a subset of $D \cap A$, while $A - D$ is a subset of $V_H^0$ $(= \{v\in V_H\colon f(H[v])=0\})$. Now, because $|A-D|>|D\cap B|$ and each vertex of $A-D$ has a~neighbor in $D\cap B$, the pigeonhole principle implies that there are two vertices $x$ and $y$ in $A-D$  which are adjacent to the same vertex in $D\cap B$. Hence, $x$ and $y$ are adjacent in $H$ (by the definition of $B_f(H)$). Now, since $f(H[x])= f(H[y])=0$, the property (2) implies that $f(H[x,y])\ge 2$. Next, since $N_{B_f(H)}((H[x,y],1))= N_{B_f(H)}((H[x,y],2))= \{x,y\}$ and $\{x,y\}\cap D=\emptyset$, the vertices $(H[x,y],1)$ and $(H[x,y],2)$ belong to $D\cap B$. Consequently, it is easy to observe that the set $D'= (D-\{(H[x,y],1), (H[x,y],2)\})\cup \{x,y\}$ is a~dominating set of $B_f(H)$, which is impossible as $|D'|=|D|$ and $|D' \cap A| > |D \cap A|$. This completes the proof.
\end{proof}

\begin{example} {\rm The graph $H$ and the function $f\colon {\cal K}_H \to \mathbb{N}$ given in Example~\ref{przyklad-1} have the properties (1) and (2) of Theorem \ref{thm:dom_main2} and therefore the bipartization $B_f(H)$ belongs to the family ${\cal B}$, that is, $\gamma(B_f(H))=|A|$, where $A$ is the smaller of two partite sets of $B_f(H)$ shown in Fig.~\ref{rys-grafu-H_f}. \\
\indent The graph $F$ and the function $\overline{g}$ given in Fig.~\ref{rys-grafu-bipatrization-dwa+} do not satisfy the condition (2) of Theorem \ref{thm:dom_main2}. However, the bipartization $G=B_{\overline{g}}(F)$ is a graph belonging to the family~${\cal B}$ since $G$ is also the bipartization $B_{\overline{f}}(H)$, with $H$ and $\overline{f}$ given in Fig.~\ref{rys-grafu-bipatrization-dwa+} and
possessing properties (1) and (2) of Theorem~\ref{thm:dom_main2}.   \\
\indent It is obvious that the complete bipartite graph $K_{m,n}$ is the bipartization of the complete graph $K_m$ (resp.\ $K_n$) with respect to the function $f\colon {\cal K}_{K_m} \to \{0,n\}$, where $f(K)=0$ if and only if $K\in {\cal K}_{K_m}-\{K_m\}$ (resp.\ $g\colon {\cal K}_{K_n} \to \{0,m\}$, where $g(K)=0$ if and only if $K\in {\cal K}_{K_n}-\{K_n\}$). It is also evident that if $\min\{m,n\}\ge 3$, then $K_{m,n}$ does not belong to the family ${\cal B}$ (as $\gamma(K_{m,n})=2<\min\{m,n\}$), and neither $K_m$ and $f$ nor $K_n$ and $g$ possess the property (2) of Theorem~\ref{thm:dom_main2}. }\end{example}

Finally, as an immediate consequence of Theorems \ref{thm:bip_tree} and \ref{thm:dom_main2} we have the following simple characterization of trees in which the domination number is equal to the size of a~smaller of its partite sets.
All such trees are bipartizations of block graphs.

\begin{cor} \label{wniosek-2} Let $T=((A,B),E_T)$ be a tree in which $1\le |A|\le |B|$. Then $\gamma(T) =|A|$ if and only if $T$ is the bipartization $B_f(H)$ of a block graph $H$ with respect to a non-zero function $f\colon {\cal K}_H \to \mathbb{N}$ and $f$ has the following two properties: \begin{itemize} \item[$(1)$] $f(K)=1$ if $K$ is a block of $H$, and $f(K')=0$ if $K'$ is a non-trivial complete subgraph of $H$ which is not a block of $H$. \item[$(2)$] $\max\{f(H[u]),f(H[v])\}\ge 1$ for every edge $uv$ of $H$
$($or,  equivalently, the set $\{v\in V_H\colon f(H[v])\ge 1\}$ is a covering set of $H$$)$. \end{itemize} \end{cor}

\end{document}